\documentclass[11pt]{article}
\usepackage{amsfonts}
\usepackage{amsmath}
\usepackage{amssymb}
\usepackage{amsthm}
\usepackage{hyperref}
\usepackage{enumerate}
\usepackage{fullpage}

\hypersetup{pdfpagemode=UseNone, pdfstartview={XYZ null null 1}, pdfpagelayout=OneColumn}

\DeclareMathAlphabet{\mathbit}{OML}{cmm}{b}{it}

\newtheorem{Theorem}{Theorem}[section]
\newtheorem{Lemma}[Theorem]{Lemma}
\newtheorem{Corollary}[Theorem]{Corollary}
\newtheorem{Observation}[Theorem]{Observation}
\newtheorem{Proposition}[Theorem]{Proposition}
\newtheorem{Example}[Theorem]{Example}
\newtheorem{Remark}[Theorem]{Remark}

\newcommand{\bv}[1]{\mathbf{#1}}
\newcommand{\brackets}[2]{\left[#1 : #2\right]}

\newcommand{\eps}{\varepsilon}
\newcommand{\mgeq}{\succeq}
\newcommand{\dcup}{\overset{\cdot}{\cup}}
\newcommand{\bigdcup}{\dot \bigcup \,}
\newcommand{\cs}[1]{\,\langle K_{#1} \rangle\,}

\newcommand{\rank}{\operatorname{rank}}
\newcommand{\range}{\operatorname{range}}
\newcommand{\mr}{\operatorname{mr}}
\newcommand{\vspan}{\operatorname{span}}

\newcommand{\xir}[1]{\xi_{[#1]}}
\newcommand{\mrr}[1]{\mr_{[#1]}}
\newcommand{\Hr}[1]{\mathcal{H}_{[#1]}}

\title{Orthogonal Representations, Projective Rank, and\break Fractional Minimum Positive Semidefinite Rank:\break Connections and New Directions}
\author{
Leslie Hogben\thanks{American Institute of Mathematics, 600 E. Brokaw Rd., San Jose, CA 95112, USA (hogben@aimath.org).}\phantom{$^*$}\thanks{Department of Mathematics, Iowa State University, Ames, IA 50011, USA (\{LHogben,kpalmow\}@iastate.edu).}
\and Kevin F. Palmowski\footnotemark[2]
\and David E. Roberson\thanks{Division of Mathematical Sciences, Nanyang Technological University, SPMS-MAS-03-01, 21 Nanyang Link, Singapore 637371 (droberson@ntu.edu.sg).} 
\and Simone Severini\thanks{Department of Computer Science, University College London, Gower Street, London WC1E 6BT, United Kingdom (simoseve@gmail.com).}
}
\date{\today}

\begin{document}

\maketitle

\begin{abstract}
Fractional minimum positive semidefinite rank is defined from $r$-fold faithful orthogonal representations and it is shown that the projective rank of any graph equals the fractional minimum positive semidefinite rank of its complement. An $r$-fold version of the traditional definition of minimum positive semidefinite rank of a graph using Hermitian matrices that fit the graph is also presented. This paper also introduces $r$-fold orthogonal representations of graphs and formalizes the understanding of projective rank as fractional orthogonal rank. Connections of these concepts to quantum theory, including Tsirelson's problem, are discussed.
\end{abstract}

{\small


\vspace{10pt}
}

\section{Introduction}

This paper deals with fractional versions of graph parameters defined by orthogonal representations, including minimum positive semidefinite rank. In Section \ref{sec-OSR-and-PR}, we extend the existing idea of an orthogonal representation for a graph via a ``higher-dimensional" construction. With this, we  introduce a new parameter, $r$-fold orthogonal rank, that is to orthogonal rank as $b$-fold chromatic number is to chromatic number (see Section \ref{sec-intro-frac} for the definition of $b$-fold chromatic number and other terms related to fractional chromatic number). This allows us to formally characterize projective rank as ``fractional orthogonal rank," a concept that was previously understood (e.g., in \cite{Robersonthesis, MR12}) but not rigorously presented  (formal definitions of projective rank and other parameters are given in Section \ref{sec-intro-defs}).

In Section \ref{sec-frac-PSD-mr}, we apply this ``fractionalization" process to the minimum positive semidefinite rank problem (viewed via faithful orthogonal representations) and develop two new graph parameters, namely, $r$-fold and fractional minimum positive semidefinite rank. We also provide an alternate definition of $r$-fold minimum positive semidefinite rank that is based on the minimum rank of a matrix that ``$r$-fits" a graph, allowing us to view the ``higher-dimensional" problem through either of the two viewpoints traditionally associated with the classical minimum positive semidefinite rank  problem.

Our final result, found in Section \ref{subsec-equality}, shows that the fractional minimum positive semidefinite rank of a graph is equal to the projective rank of the complement of the graph. This result serves to connect the two seemingly different problems; moving forward, this will allow the extensive existing literature on minimum positive semidefinite rank to be used to inform new developments in the more recently introduced area of projective rank.

In the remainder of this introduction we discuss applications of the fractional parameters discussed (Section \ref{sec-intro-apps}), give a brief introduction to the fractional approach to chromatic number to motivate our definitions (Section  \ref{sec-intro-frac}), and provide necessary notation and terminology (Section \ref{sec-intro-defs}).

\subsection{Applications} \label{sec-intro-apps}

Linear algebraic structures and associated graph theoretic frameworks  have  recently become more important tools to
study the fundamental differences that characterize theories of nature, like
classical mechanics, quantum mechanics, and general probabilistic theories.
Matrices, graphs, and their related combinatorial optimization techniques turn out to
provide a surprisingly general language with which to approach questions connected with
foundational ideas, such as the analysis of contextual inequalities and non-local
games \cite{cabello, cameron}, and with concrete aspects, such as quantifying
various capacities of entanglement-assisted channels \cite{duan, leung}, and
the overhead needed to classically simulate quantum computation \cite{emerson}.

A point of strength of such frameworks is their ability to reformulate
mathematical questions in a coarser manner that is nonetheless effective, in
some cases, to single out specific facts. Tsirelson's problem \cite{tsirelson}
provides a remarkable example: deciding whether the mathematical models of
non-relativistic quantum mechanics, where observers have linear operators acting on a
finite dimensional tensor product space, and algebraic quantum field theory,
where observers have commuting linear operators on a single (possibly infinite dimensional) space, produce the same set of correlations. We know that if Tsirelson's
problem has a positive answer then the notorious Connes' Embedding conjecture
\cite{capraro, ozawa}, originally concerned with an approximation property for
finite von Neumann algebras, is true.

Tsirelson's problem can be seen from a combinatorial matrix point of view by working with
graphs and their associated algebraic structures \cite{paulsen1}. Roughly
speaking, instead of constructing sets of correlation matrices, we can try
looking for various patterns of zeroes in the sets, as in the spirit of
combinatorial matrix theory. The projective rank, denoted $\xi_{f}$, 
is a
recently introduced graph parameter with the potential for settling the above
discussion. Indeed, it has been shown that if there exists a graph whose
projective rank is irrational, then Tsirelson's problem has a negative answer
\cite{paulsen2}.

Projective representations and projective rank were originally defined in \cite{MR12} as a tool for studying quantum colorings and quantum homomorphisms of graphs. Quantum colorings and the quantum chromatic number give quantitative measures of the advantage that quantum entanglement provides in performing distributed tasks and in distinguishing scenarios related to classical and quantum physics, respectively. In fact, the existence of a quantum $n$-coloring for a given graph is equivalent to the existence of a projective representation of value $n$ for the Cartesian product of the graph with a complete graph on $n$ vertices.

It was also shown in \cite{MR12} that projective rank is monotone with respect to quantum homomorphisms, i.e., if there exists a quantum homomorphism from a graph $G$ to a graph $H$, then $\xi_f(G) \leq \xi_f(H)$. This shows that projective rank is a  lower bound for quantum chromatic number, and more generally provides a method for forbidding the existence of quantum homomorphisms. Indeed, this approach was used to determine the quantum odd girth of the Kneser graphs in \cite{Robersonthesis}. Projective rank has also been studied from a purely graph theoretic point of view, and in \cite{Stahlke14} it was shown that this parameter is multiplicative with respect to the lexicographic and disjunctive graph products. Using this fact the authors were able to find a separation between quantum chromatic number and a recently defined semidefinite relaxation of this parameter, answering a question posed in \cite{paulsen1}.

This paper takes a linear algebraic approach to these questions, building connections between recent graph theoretical approaches to quantum questions and existing literature on orthogonal representations and minimum positive semidefinite rank.

\subsection{A fractional approach} \label{sec-intro-frac}

To demonstrate the fractional approach that we use with orthogonal representations and minimum positive semidefinite rank, consider the following derivation of the fractional chromatic number as found in \cite{fgt}. The \emph{chromatic number} $\chi(G)$ of a graph $G$ is the least number $c$ such that $G$ can be colored with $c$ colors; that is, we can assign to each vertex of $G$ one of $c$ colors in such a way that adjacent vertices receive different colors. A coloring with $c$ colors can be generalized to a \emph{$b$-fold coloring with $c$ colors}, or a $c$:$b$-coloring: from a palette of $c$ colors, assign $b$ colors to each vertex of $G$ such that adjacent vertices receive disjoint sets of colors. For a fixed $b$, the \emph{$b$-fold chromatic number} of $G$, $\chi_b(G)$, is the smallest $c$ such that G has a $c$:$b$-coloring. With this, we can define the \emph{fractional chromatic number} of $G$ as
\[ \chi_f(G) = \inf_b \frac{\chi_b(G)}{b} \, . \]
While it is not obvious, it can be shown that $\chi_f(G)$ is always a rational number, as there is an alternative linear programming formulation for the parameter for which strong duality holds. For further information on fractional coloring, including a time-scheduling interpretation of the problem, see the discussions in the Preface and Chapter 3 of \cite{fgt}.

The process of assigning objects to the vertices of a graph, subject to certain constraints, is a key element common to the problems we examine in this work, and the procedure of generalizing from assigning one object to assigning $b$-many objects (or, in our case, $b$-dimensional or rank-$b$ objects) is an underlying theme. At each stage of the process, we are interested in graph parameters that give information about the ``most efficient" set of objects we can use, with the end goal of developing fractional versions of existing parameters (in the spirit of \cite{fgt}) and connecting the more recent work on projective rank with existing ideas from the realm of minimum positive semidefinite rank.

Rather than the colors used for coloring problems, the objects that we assign to the vertices of a graph are vectors and matrices, which adds a distinctly linear algebraic flavor to both the problems and the constraints: the idea of ``different colors" translates to orthogonality conditions on our objects. As such, our results often see linear algebra and graph theory working hand-in-hand, with structure found in one discipline influencing results that are based in the other.

\subsection{Background, definitions, and notation} \label{sec-intro-defs}

The natural numbers, $\mathbb{N}$, start at 1. We use the notation $\brackets{a}{b}$ to denote the set of integers $\{a, a+1, \ldots, b-1, b\}$. Throughout, $d$ and $r$ are used to represent natural numbers. Vectors are denoted by boldface font, typically $\bv{x}$, and matrices are capital letters, typically $A$, $B$, $P$, or $X$, depending on context. The symbol $0$ denotes either the scalar zero or a zero matrix, and an identity matrix is denoted by $I$; any of these may be subscripted to clarify their sizes. We follow the usual convention of denoting the $j^{th}$ standard basis vector in $\mathbb{C}^d$ (for some $d$) as $\bv{e}_j$. Rows and columns of matrices may be indexed either by natural numbers or by vertices of a graph, depending on context. The elements of a matrix $A$ are denoted $a_{ij}$; if $A$ is a block matrix, then its blocks are denoted $A_{ij}$. Graphs are usually denoted by $G$ or $H$, vertices by $u, v$ or $i, j$, and edges by $uv$ or $ij$.

If $A \in \mathbb{C}^{p \times p}$ and $B \in \mathbb{C}^{q \times q}$, then the \emph{direct sum} of $A$ and $B$, denoted $A \oplus B$, is the block diagonal matrix 
\[ \left[ \begin{array}{cc} A & 0 \\ 0 & B \end{array} \right] \in \mathbb{C}^{(p+q) \times (p+q)}. \]
We denote the conjugate transpose of $A$ by $A^*$. A \emph{Hermitian} matrix satisfies $A = A^*$. A Hermitian matrix $A \in \mathbb{C}^{n \times n}$ is \emph{positive semidefinite}, denoted $A \mgeq 0$, if $\bv{x}^* A \bv{x} \geq 0$ for all $\bv{x} \in \mathbb{C}^n$, or equivalently, if all of its eigenvalues are nonnegative.

Typically, $G = (V, E)$ will denote a simple undirected graph on $n$ vertices, where $V = V(G)$ is the set of vertices of $G$ and $E = E(G)$ is the set of edges of $G$. An \emph{isolated vertex} is a vertex that is not adjacent to any other vertex of $G$. A \emph{subgraph} of a graph $G$ is a graph $H$ such that $V(H) \subseteq V(G)$ and $E(H) \subseteq E(G)$. An \emph{induced subgraph} of a graph $G$, denoted $G[W]$ for some set $W \subseteq V(G)$, is a subgraph with vertex set $W$ such that if $u,v \in W$ and $uv \in E(G)$, then $uv \in E(G[W])$. The \emph{union} of graphs $G$ and $H$, denoted $G \cup H$, is the graph with vertex set $V(G \cup H) = V(G) \cup V(H)$ and edge set $E(G \cup H) = E(G) \cup E(H)$. If $V(G) \cap V(H) = \emptyset$, then this union is \emph{disjoint} and denoted $G \dcup H$. The \emph{complement} of $G$, denoted $\overline{G}$, is the graph with $V(\overline{G}) = V(G)$ and $E(\overline{G}) = \{ uv : u \neq v, uv \notin E(G) \}$. An \emph{independent set} in $G$ is a set $W \subseteq V(G)$ such that if $u, v \in W$, then $uv \notin E(G)$. The \emph{independence number} of $G$, denoted $\alpha(G)$, is the largest possible cardinality of an independent set in $G$. A \emph{clique} in $G$ is an induced subgraph $H$ that is a complete graph, i.e., $uv \in E(H)$ for every $u,v \in V(H)$. The \emph{clique number} of $G$, denoted $\omega(G)$, is the largest possible order of a clique in $G$. A \emph{clique-sum} of graphs $G$ and $H$ on $K_t$, i.e., the graph $G \cup H$ where $G \cap H = K_t$, is denoted by $G \cs{t} H$; this is also called a $t$-clique-sum of $G$ and $H$. A \emph{chordal graph} is a graph that does not have any induced cycles of length greater than 3; any chordal graph can be constructed as clique-sum(s) of complete graphs. A \emph{perfect graph} is a graph $G$ for which every induced subgraph $H$ of $G$ satisfies $\omega(H) = \chi(H)$. A \emph{cut-vertex} of a connected graph $G$ is a vertex whose deletion disconnects $G$. A graph with a cut-vertex can be viewed as a $1$-clique-sum.

We work in the vector space $\mathbb{C}^d$ for some $d \in \mathbb{N}$. We use $S$ to denote a subspace of a vector space. A \emph{basis matrix} for an $r$-dimensional subspace $S$ of $\mathbb{C}^d$ is a matrix $X \in \mathbb{C}^{d \times r}$ that has orthonormal columns and satisfies $S = \range(X)$. We say that two subspaces $S_1$ and $S_2$ of $\mathbb{C}^d$ are \emph{orthogonal}, denoted $S_1 \perp S_2$, if $\bv{u}_1^* \bv{u}_2 = 0$ for all $\bv{u}_1 \in S_1$ and all $\bv{u}_2 \in S_2$; an equivalent condition is that $X_1^* X_2 = 0$, where $X_1$ and $X_2$ are basis matrices for $S_1$ and $S_2$, respectively.

Given some graph $G$ and $d \in \mathbb{N}$, an \emph{orthogonal representation} in $\mathbb{C}^d$ for $G$ is a set of unit vectors $\{\bv{x}_u\}_{u \in V(G)} \subset \mathbb{C}^d$ such that $\bv{x}_u^* \bv{x}_v = 0$ if $uv \in E(G)$. It is clear that such a representation always exists for $d = |V(G)|$. Provided that $G$ has at least one edge, it is clear that such a representation cannot be made for $d = 1$. We define the \emph{orthogonal rank} of $G$ to be
\[ \xi(G) = \min\left\{ d : G \text{ has an orthogonal representation in $\mathbb{C}^d$} \right\}. \]

Let $d, r \in \mathbb{N}$ with $r \leq d$. A \emph{$d/r$-projective representation}, or $d/r$-representation, is an assignment of matrices $\{P_u\}_{u \in V(G)}$ to the vertices of $G$ such that
\begin{itemize}
\item for each $u \in V(G)$, $P_u \in \mathbb{C}^{d \times d}$, $\rank P_u  = r$, $P_u^* = P_u$, and $P_u^2 = P_u$; and
\item if $uv \in E(G)$, then $P_u P_v = 0$.
\end{itemize}
In words, a $d/r$-representation is an assignment of rank-$r$ $(d \times d)$ orthogonal projection matrices (projectors) to the vertices of $G$ such that adjacent vertices receive projectors that are orthogonal. The \emph{projective rank} of $G$ is defined as
\[ \xi_f(G) = \inf_{d,r} \left\{ \frac{d}{r} : G \text{ has a $d/r$-representation} \right\}. \]
Projective rank was first introduced in 2012 by Roberson and Man\v cinska, where it is noted that $\xi_f(G) \leq \xi(G)$; see \cite{Robersonthesis} and \cite{MR12} for additional information, properties, and applications.

Complementary to the idea of an orthogonal representation is that of a \emph{faithful} orthogonal representation (here we follow the complementary usage in the minimum rank literature). In order for the definitions given next to coincide with those in the minimum rank literature, we must assume that the graph $G$ has no isolated vertices. A \emph{faithful orthogonal representation} in $\mathbb{C}^d$ for a graph $G$ is a set of unit vectors $\{ \bv{x}_u \}_{u \in V(G)} \subset \mathbb{C}^d$ such that $\bv{x}_u^*\bv{x}_v = 0$ if and only if $uv \notin E(G)$. We define the \emph{minimum positive semidefinite rank} of $G$ as
\begin{equation} \label{def-mrplus}
\mr^+(G) = \min \left\{ d : G \text{ has a faithful orthogonal representation in $\mathbb{C}^d$} \right\}.
\end{equation}

We say that a matrix $A \in \mathbb{C}^{n \times n}$ \emph{fits} the order-$n$ graph $G$ if $a_{ii} = 1$ for all $i \in \brackets{1}{n}$, and for all $i \neq j$, we have $a_{ij} = 0$ if and only if $ij \notin E(G)$. Let $\mathcal{H}^+ (G) = \left\{ A \in \mathbb{C}^{n \times n} : A \mgeq 0 \text{ and $A$ fits $G$} \right\}$. A faithful orthogonal representation in $\mathbb{C}^d$ for $G$ corresponds to a matrix $A \in \mathcal{H}^+(G)$ with $\rank A \leq d$, and a matrix $A \in \mathcal{H}^+(G)$ with rank $d$ can be factored as $A = B^* B$ for some $B \in \mathbb{C}^{d \times n}$. Thus an alternate characterization (see, e.g., \cite{HLA2ch46}) of $\mr^+(G)$ is
\[ \mr^+(G) = \min\{ \rank A : A \in \mathcal{H}^+(G) \}, \]
(and in fact, this is the customary definition of this parameter).

The definitions and explanation given here coincide with those in the literature provided that the graph $G$ has no isolated vertices. The most common definition of $\mathcal{H}^+(G)$ in the literature does not contain the assumption that $a_{ii}=1$. If vertex $i$ is adjacent to at least one other vertex, then properties of positive semidefinite matrices require $a_{ii} > 0$, and so $A$ can be scaled by a positive diagonal congruence to a matrix of the same rank and nonzero pattern that has all diagonal entries equal to one. However, consider the case where $G$ consists of $n$ isolated vertices (no edges): then as defined in \cite{BHH11, HLA2ch46}, etc., $\mr^+(G) = 0$, whereas with our definition $\mr^+(G) = n$. The two definitions of minimum positive semidefinite rank coincide precisely when $G$ has no isolated vertices. Our definition facilitates connections to the use of orthogonal rank in the study of quantum issues, and the assumption of no isolated vertices is needed only when connecting to the minimum rank literature, so we omit it except when discussing connections to such work (where we state either this assumption or one that implies it, such as the graph being connected and of order at least two). We also note that for any graph the values of the parameters studied can be computed from their values on the connected components of the graph (see Section \ref{sec-frac-PSD-mr}), which facilitates handling cases with isolated vertices separately.


\section{Orthogonal subspace representations and projective rank} \label{sec-OSR-and-PR}

In this section, we introduce and discuss $(d;r)$ orthogonal subspace representations for a graph $G$, which are extensions of orthogonal representations in the spirit of fractional graph theory \cite{fgt}. The $r$-fold orthogonal rank of a graph, $\xir{r}(G)$, is defined and some properties of this quantity are examined. We then relate these representations to $d/r$-projective representations and tie projective rank into the new theory, formalizing the existing understanding that projective rank and ``fractional orthogonal rank" are one and the same.

Unless otherwise specified, all matrices and vectors in this section are assumed to be complex-valued.

\subsection{Orthogonal subspace representations and $\mathbit{r}$-fold orthogonal rank}

Let $G$ be a graph and let $d, r \in \mathbb{N}$ with $d \geq r$. A \emph{$(d;r)$ orthogonal subspace representation}, or $(d;r)$-OSR, for $G$ is a set of subspaces $\{S_u\}_{u \in V(G)}$ such that 
\begin{itemize}
\item for each $u \in V(G)$, $S_u$ is an $r$-dimensional subspace of $\mathbb{C}^d$; and
\item if $uv \in E(G)$, then $S_u \perp S_v$.
\end{itemize}

The \emph{$r$-fold orthogonal rank} of a graph $G$ is defined by
\[ \xir{r}(G) = \min \left\{ d : G \text{ has a $(d;r)$ orthogonal subspace representation} \right\}. \]
An orthogonal representation in $\mathbb{C}^d$ naturally generates a $(d;1)$ orthogonal subspace representation, and vice versa, so $\xi(G) = \xir{1}(G)$.

We now explore some properties of $\xir{r}(G)$.

\begin{Lemma} \label{xirsubadditive}
$\xir{r}$ is a subadditive function of $r$, i.e., for every graph $G$ and all $r, s \in \mathbb{N}$, 
\[ \xir{r+s}(G) \leq \xir{r}(G) + \xir{s}(G). \]
\end{Lemma}
\begin{proof}
Let $d_r = \xir{r}(G)$ and $d_s = \xir{s}(G)$. Then $G$ has a $(d_r ;r)$ orthogonal subspace representation containing $r$-dimensional subspaces of $\mathbb{C}^{d_r}$, say $\{S^r_u\}_{u \in V(G)}$, and a $(d_s; s)$ orthogonal subspace representation containing $s$-dimensional subspaces of $\mathbb{C}^{d_s}$, say $\{S^s_u\}_{u \in V(G)}$. We show by construction that there exists an orthogonal subspace representation for $G$ containing $(r+s)$-dimensional subspaces of $\mathbb{C}^{d_r + d_s}$.

For each $u \in V(G)$, let $X_u^r \in \mathbb{C}^{d_r \times r}$ and $X_u^s \in \mathbb{C}^{d_s \times s}$ be basis matrices for $S_u^r$ and $S_u^s$, respectively. Define
\[ X_u = \left[ \begin{array}{cc} X_u^r & 0_{d_r \times s} \\ 0_{d_s \times r} & X_u^s \end{array} \right] \in \mathbb{C}^{(d_r+d_s) \times (r+s)} \]
and let $S_u = \range(X_u)$. We immediately see that $S_u$ is a subspace of $\mathbb{C}^{d_r + d_s}$, $X_u$ is a basis matrix for $S_u$, and $\dim(S_u) = \rank X_u = \rank X_u^r + \rank X_u^s = r+s$.

Suppose $u, v \in V(G)$ and let $X_u^r$, $X_v^r$, $X_u^s$, $X_v^s$, $X_u$, and $X_v$ be as above; then
\[ X_u^* X_v = \left[ \begin{array}{cc}
	(X_u^r)^*(X_v^r) & 0 \\
	0 &  (X_u^s)^*(X_v^s)
\end{array} \right]. \]
Suppose $uv \in E(G)$. Since $\{S_u^r\}$ is an orthogonal subspace representation, we have $(X_u^r)^*(X_v^r) = 0$; similarly, $(X_u^s)^*(X_v^s) = 0$, so $X_u^* X_v = 0$. Since $X_u$ and $X_v$ are basis matrices for $S_u$ and $S_v$, respectively, we conclude that if $uv \in E(G)$, then $S_u \perp S_v$.

Thus $\{S_u\}_{u \in V(G)}$ is a $(d_r + d_s; r+s)$ orthogonal subspace representation for $G$, so $\xir{r+s}(G) \leq d_r + d_s = \xir{r}(G) + \xir{s}(G)$. \qquad
\end{proof}

\begin{Corollary}
For every graph $G$ and all $r \in \mathbb{N}$, $\frac{\xir{r}(G)}{r} \leq \xi(G)$.
\end{Corollary}
\begin{proof}
Since $\xir{1}(G) = \xi(G)$, we have
\[ \xir{r}(G) \leq \xir{r-1}(G) + \xi(G) \leq \ldots \leq r \cdot \xi(G) . \qedhere\]
\end{proof}

\begin{Observation} \label{xir-geq-omega}
For every graph $G$ and all $r \in \mathbb{N}$, $\xir{r}(G) \geq r \cdot \omega(G)$.
\end{Observation}

\begin{Proposition} \label{xir-subgraph}
Let $r \in \mathbb{N}$ and let $H$ be a subgraph of $G$. Then $\xir{r}(H) \leq \xir{r}(G)$.
\end{Proposition}
\begin{proof}
Since every edge of $H$ is an edge of $G$, any $(d;r)$ orthogonal subspace representation for $G$ provides a $(d;r)$ orthogonal subspace representation for $H$, and the result is immediate. \qquad
\end{proof}

\begin{Proposition} \label{xirdisjointmax}
Suppose $r \in \mathbb{N}$ and $G = \bigdcup_{i=1}^t G_i$ for some graphs $\{G_i\}_{i=1}^t$. Then $\xir{r}(G) = \max_i \left\{ \xir{r}(G_i) \right\}$.
\end{Proposition}
\begin{proof}
Since each $G_i$ is an induced subgraph of $G$, we have $\xir{r}(G_i) \leq \xir{r}(G)$ for each $i$, so $\max_i \left\{ \xir{r}(G_i) \right\} \leq \xir{r}(G)$.

\smallskip
For each $i \in \brackets{1}{t}$, let $d_i = \xir{r}(G_i)$ and let $d = \max_i \{d_i\}$.  Let $\{S_u^i\}_{u \in V(G_i)}$ be a $(d_i; r)$ orthogonal subspace representation for $G_i$ and for each vertex $u \in V(G_i)$ let $X_u^i \in \mathbb{C}^{d_i \times r}$ be a basis matrix for $S_u^i$. For each $u \in V(G)$, we have $u \in V(G_i)$ for some $i$; define
\[ S_u = \range \left[ \begin{array}{c} X_u^i \\ 0_{(d-d_i) \times r} \end{array} \right]. \]

Each $S_u$ is an $r$-dimensional subspace of $\mathbb{C}^d$, and if $uv \in E(G)$, then $uv \in E(G_k)$ for some $k$, so $S_u^k \perp S_v^k$, which implies that $S_u \perp S_v$ (by construction). Therefore, $\{ S_u \}_{u \in V(G)}$ is a $(d;r)$-OSR for $G$, so $\xir{r}(G) \leq d = \max_i \{ \xir{r}(G_i) \}$ and equality follows. \qquad
\end{proof}

This result does not hold for arbitrary graph unions, as the following example for the $r = 1$ case shows.

\begin{Example} {\rm
Let $G = C_5$ with $V(G) = \{1,2,3,4,5\}$ and $E(G) = \{12, 23, 34, 45, 51\}$. Define $G_1 = P_4$ with $V(G_1) = \{1,2,3,4\}$ and $E(G_1) = \{12, 23, 34 \}$ and define $G_2 = P_3$ with $V(G_2) = \{4,5,1\}$ and $E(G_2) = \{45, 51\}$. We see that $G = G_1 \cup G_2$, but since $\xi(P_3) = \xi(P_4) = 2$ and $\xi(C_5) = 3$, it is not true that $\xi(G) = \max\{ \xi(G_1), \, \xi(G_2)\}$.
}
\end{Example}

While the maximum property observed in Proposition \ref{xirdisjointmax} may not carry over to the case when $G$ is a nondisjoint union of graphs, we are still able to obtain a weaker result, which follows.

\begin{Proposition}
Suppose $r \in \mathbb{N}$ and $G = \bigcup_{i=1}^t G_i$, where $G_i$ is an induced subgraph of $G$ for each $i$. Then $\xir{r}(G) \leq \sum_{i=1}^t \xir{r}(G_i)$.
\end{Proposition}
\begin{proof}
We prove the result for the case where $t = 2$ and note that recursive application of this case will prove the more general one.

For each $i \in \{1,2\}$, let $d_i = \xir{r}(G_i)$ and $\{S_u^i\}_{u \in V(G_i)}$ be a $(d_i;r)$-OSR for $G_i$, and for each $u \in V(G_i)$, let $X_u^i \in \mathbb{C}^{d_i \times r}$ be a basis matrix for $S_u^i$.

We partition $V(G) = V(G_1) \cup V(G_2)$ into three disjoint sets and consider vertices in each set. If $u \in V(G_1) \setminus V(G_2)$, let
\[ X_u = \left[ \begin{array}{c} X_u^1 \\ 0_{d_2 \times r} \end{array} \right]; \]
if $u \in V(G_2) \setminus V(G_1)$, let
\[ X_u = \left[ \begin{array}{c} 0_{d_1 \times r} \\ X_u^2 \end{array} \right]; \]
and if $u \in V(G_1) \cap V(G_2)$, let
\[ X_u = \left[ \begin{array}{c} X_u^1 \\ X_u^2 \end{array} \right]. \]
For each $u \in V(G)$, let $S_u = \range(X_u)$. Each $S_u$ is an $r$-dimensional subspace of $\mathbb{C}^{d_1+d_2}$.

We consider multiple cases to show that if $uv \in E(G)$, then $X_u^* X_v = 0$, so $S_u \perp S_v$. Throughout, we assume that $uv \in E(G)$.

First, suppose that $u \in V(G_1) \setminus V(G_2)$; then either $v \in V(G_1) \setminus V(G_2)$ or $v \in V(G_1) \cap V(G_2)$. In either case, $uv \in E(G_1)$ (since $G_1$ is an induced subgraph), and block multiplication yields $X_u^* X_v = (X_u^1)^* X_v^1$. Since $S_u^1 \perp S_v^1$, this quantity equals the zero matrix, so $S_u \perp S_v$. The case where $u \in V(G_2) \setminus V(G_1)$ is similar.

If $u, v \in V(G_1) \cap V(G_2)$, then $uv \in E(G_1) \cap E(G_2)$ since $G_1$ and $G_2$ are induced subgraphs. Then $X_u^* X_v = (X_u^1)^* X_v^1 + (X_u^2)^* X_v^2$. Since $S_u^1 \perp S_v^1$ and $S_u^2 \perp S_v^2$, this quantity is again the zero matrix, so $S_u \perp S_v$.

Therefore, $\{S_u\}_{u \in V(G)}$ is a $(d_1+d_2;r)$-OSR for $G$, so $\xir{r}(G) \leq d_1 + d_2 = \xir{r}(G_1) + \xir{r}(G_2)$. \qquad
\end{proof}

\begin{Lemma} \label{stdcliquerep}
Suppose that the complete graph $K_t$ is a subgraph of $G$ with $V(K_t) = \brackets{1}{t}$ and $G$ has a $(d;r)$ orthogonal subspace representation. Then $d \geq rt$ and $G$ has a $(d;r)$ orthogonal subspace representation in which the vertex $i \in V(K_t)$ is represented by
\[ \vspan \left\{ \bv{e}_{(i-1)r+1}, \ldots , \bv{e}_{(i-1)r+r-1},\bv{e}_{ir} \right\}. \]
\end{Lemma}
\begin{proof}
By Observation \ref{xir-geq-omega}, $d \geq r \cdot \omega(G) \geq rt$.

If $M \in \mathbb{C}^{d \times \ell}$ for some $\ell \leq d$ and the columns of $M$ are orthonormal, then by a change of orthonormal basis there exists a unitary matrix $U \in \mathbb{C}^{d \times d}$ such that $UM = [ \bv{e}_{1}, \ldots, \bv{e}_{\ell}]$.

Let $\{S_u\}_{u \in V(G)}$ be a $(d;r)$ orthogonal subspace representation for $G$ and for each $u \in V(G)$ let $X_u$ be a basis matrix for $S_u$. Define $M = [X_1, \ldots, X_t]$ and choose $U$ so that $UM=[ \bv{e}_1 , \ldots, \bv{e}_{tr}]$.  Define $S_u' = \range(UX_u)$. Then $\left\{ S_u' \right\}_{u\in V(G)}$ is a $(d;r)$ orthogonal subspace representation for $G$ with the desired property. \qquad
\end{proof}

\begin{Theorem}
If $G = G_1 \cs{t} G_2$ and $r \in \mathbb{N}$, then $\xir{r}(G) = \max \left\{ \xir{r}(G_1), \xir{r}(G_2) \right\}$.
\end{Theorem}
\begin{proof}
Without loss of generality, let $d_1 = \xir{r}(G_1) \geq d_2 = \xir{r}(G_2)$ and $V(K_t) = \brackets{1}{t}$. Then by Lemma \ref{stdcliquerep}, for $i = 1,2$, each $G_i$ has a $(d_1; r)$ orthogonal subspace representation, $\{S_u^i\}_{u\in V(G)}$, in which vertex $v \leq t$ is represented by $S_v^i = \vspan \left\{ \bv{e}_{(v-1)r+1}, \ldots, \bv{e}_{(v-1)r+r-1}, \bv{e}_{vr} \right\}$.  Thus for $v \in \brackets{1}{t}$, $S_v^1 = S_v^2$; denote this common subspace by $S_v$.

For vertices $u\in V(G_i) \setminus \brackets{1}{t}$, define $S_u = S_u^i$ (observe that $u > t$ is in only one of $V(G_1)$ or $V(G_2)$). Then $\{S_u\}_{u\in V(G)}$ is a $(d_1;r)$ orthogonal subspace representation for $G$. \qquad
\end{proof}

\begin{Proposition} \label{prop-xir-equals-romega}
If $G$ is a graph with $\omega(G) = \chi(G)$, then $\xir{r}(G) = r \cdot \omega(G)$ for every $r \in \mathbb{N}$.
\end{Proposition}
\begin{proof}
It is well-known that $\xi(G) \leq \chi(G)$ (see, e.g., \cite{Robersonthesis}). Therefore,
\[ r \cdot \omega(G) \leq \xir{r}(G) \leq r \cdot \xi(G) \leq r \cdot \chi(G) = r \cdot \omega(G) \]
and thus equality holds throughout.
\end{proof}

We note that perfect graphs and chordal graphs are among those that satisfy $\omega(G) = \chi(G)$, and so Proposition \ref{prop-xir-equals-romega} applies to these classes.

\begin{Remark} {\rm
Since $\xir{1}(G) = \xi(G)$ for every graph $G$, the previous properties of $r$-fold orthogonal rank also apply to orthogonal rank, where appropriate.
}
\end{Remark}

\subsection{Projective rank as fractional orthogonal rank}

It is easy to see that $(d;r)$ orthogonal subspace representations are closely related to $d/r$-\break representations; in fact, they are in one-to-one correspondence.

\begin{Proposition} \label{drOSRiffdr}
A graph $G$ has a $(d;r)$ orthogonal subspace representation if and only if $G$ has a $d/r$-representation.
\end{Proposition}
\begin{proof}
Suppose that $G$ has a $(d;r)$ orthogonal subspace representation $\{S_u\}_{u \in V(G)}$, so each $S_u$ is an $r$-dimensional subspace of $\mathbb{C}^d$. For each $u \in V(G)$, define $P_u =X_u X_u^*$, where $X_u \in \mathbb{C}^{d \times r}$ is a basis matrix for $S_u$. It is then easy to verify that $P_u \in \mathbb{C}^{d \times d}$, $\rank P_u  = \rank X_u  = r$, $P_u^* = P_u$, and $P_u^2 = P_u$.

Let $uv \in E(G)$, so $S_u \perp S_v$. We see that
\[ S_u \perp S_v \iff X_u^* X_v = 0 \iff X_u X_u^* X_v X_v^* = 0 \iff P_u P_v = 0. \]
Thus if $uv \in E(G)$, then $P_u P_v = 0$. We conclude that $\{P_u\}_{u \in V(G)}$ is a $d/r$-representation for $G$.

For the converse, suppose that $\{P_u\}_{u \in V(G)}$ is a $d/r$-representation for $G$. For each $u \in V(G)$, let $P_u = X_u I_r X_u^*$ be a reduced singular value decomposition of the projector $P_u$ (where $X_u \in \mathbb{C}^{d \times r}$) and define $S_u = \range(P_u) = \range(X_u)$. Clearly $S_u$ is an $r$-dimensional subspace of $\mathbb{C}^d$. If $uv \in E(G)$, then $P_u P_v = 0$, so by the above chain of equivalences $S_u \perp S_v$. Therefore, $\{S_u\}_{u \in V(G)}$ is a $(d;r)$ orthogonal subspace representation for $G$. \qquad
\end{proof}

With this in mind, we obtain the following ``fractional" definition of projective rank.

\begin{Theorem} \label{xifandxir} For every graph $G$,
\[ \xi_f(G) = \inf_r \left\{ \frac{\xir{r}(G)}{r} \right\}. \]
\end{Theorem}

\begin{proof}
\begin{align*}
\inf_r \left\{ \frac{\xir{r}(G)}{r} \right\}
	&=	\inf_r \left\{ \frac{ \min\{ d : G \text{ has a $(d;r)$-OSR} \}}{r}  \right\} \\
	&=	\inf_r \left\{ \min_d \left\{ \frac{d}{r} : G \text{ has a $(d;r)$-OSR} \right\} \right\} \\
	&=	\inf_{d,r} \left\{ \frac{d}{r} : G \text{ has a $(d;r)$-OSR} \right\} \\
	&=	\inf_{d,r} \left\{ \frac{d}{r} : G \text{ has a $d/r$-representation} \right\} \\
	&=	\xi_f(G). \qedhere
\end{align*}
\end{proof}

Given that this expression of $\xi_f(G)$ is similar to that of $\chi_f(G)$ given in \cite{fgt}, it is not unreasonable to hope that this could shed some light on the question of the rationality of $\xi_f(G)$ for all graphs.\footnote{Recall that $\chi_f(G)$ is rational for any graph $G$.} Unfortunately, finding a $b$-fold coloring with $c$ colors for $G$ is ultimately a far different problem from finding a $(d;r)$ orthogonal subspace representation for $G$. In the $b$-fold coloring problem, we have a restriction on the number of available colors, which adds a certain finiteness to the problem: each vertex is assigned a subset of the available $c < \infty$ colors. In contrast, restricting the subspaces to lie in $\mathbb{C}^d$ in the orthogonal subspace representation problem does not impose this same type of finiteness: each vertex is assigned a finite dimensional subspace of $\mathbb{C}^d$, and $d < \infty$, but there are infinitely many subspaces that can be assigned to each vertex.

We provide one additional equivalent definition of projective rank, for which we need the following utility result from \cite{fgt}, also commonly known as Fekete's Lemma.

\begin{Lemma}[\cite{fgt}, Lemma A.4.1] \label{fgtA41}
Suppose $g : \mathbb{N} \to \mathbb{R}$ is subadditive and $g(n) \geq 0$ for all $n$. Then the limit
\[ \lim_{n \to \infty} \frac{g(n)}{n} \]
exists and is equal to the infimum of $g(n) / n$ $(n \in \mathbb{N})$.
\end{Lemma}

Since $\xir{r}$ is subadditive, this yields the following corollary to the previous theorem.

\begin{Corollary} \label{xirinfequalslim} For every graph $G$,
\[ \xi_f(G) = \inf_r \left\{ \frac{\xir{r}(G)}{r} \right\} = \lim_{r \to \infty} \frac{\xir{r}(G)}{r} \, , \]
and this limit exists.
\end{Corollary}

With this result, we see that many of the properties of $\xir{r}(G)$ also apply to $\xi_f(G)$.

\begin{Theorem} \label{xifproperties}
For every graph $G$:
\begin{enumerate}[i)]
\item {\rm \cite{Robersonthesis, MR12} } $\xi_f(G) \geq \omega(G)$.
\item If $H$ is a subgraph of $G$, then $\xi_f(H) \leq \xi_f(G)$.
\item If $G = \bigdcup_{i=1}^t G_i$ for some graphs $\{ G_i \}_{i=1}^t$, then $\xi_f(G) = \max_i \left\{ \xi_f(G_i) \right\}$.
\item If $G = \bigcup_{i=1}^t G_i$ for some induced subgraphs $\{ G_i \}_{i=1}^t$, then $\xi_f(G) \leq \sum_{i=1}^t \xi_f(G_i)$.
\item If $G = G_1 \cs{t} G_2$, then $\xi_f(G) = \max \left\{ \xi_f(G_1), \xi_f(G_2) \right\}$.
\item If $G$ satisfies $\omega(G) = \xi(G)$, then $\xi_f(G) = \omega(G)$.
\end{enumerate}
\end{Theorem}
\begin{proof}
Consider the second claim. By Proposition \ref{xir-subgraph}, for any $r \in \mathbb{N}$, $\xir{r}(H) \leq \xir{r}(G)$, so $\frac{\xir{r}(H)}{r} \leq \frac{\xir{r}(G)}{r}$. Taking the limit as $r$ approaches $\infty$ and applying Corollary \ref{xirinfequalslim}, we have $\xi_f(H) \leq \xi_f(G)$.

The remaining claims follow by applying similar arguments to the corresponding $r$-fold results.
\end{proof}


\section{Fractional minimum positive semidefinite rank} \label{sec-frac-PSD-mr}

In this section, we introduce $(d;r)$ faithful orthogonal subspace representations, $r$-fold minimum positive semidefinite rank, and fractional minimum positive semidefinite rank, extending the definitions of faithful orthogonal representations and minimum positive semidefinite rank. We then introduce faithful $d/r$-projective representations and connect everything to projective rank. A connection to positive semidefinite matrices is explored, and properties of our new quantities are proven.

Unless otherwise specified, all matrices and vectors in this section are assumed to be complex-valued (the literature on minimum positive semidefinite rank is mixed, with both real and complex cases studied).

\subsection{Faithful orthogonal subspace representations and fractional minimum positive semidefinite rank}

Given a graph $G$ and $d, r \in \mathbb{N}$ with $r \leq d$, a \emph{$(d;r)$ faithful orthogonal subspace representation}, or $(d;r)$-FOSR, for $G$ is a set of subspaces $\{S_u\}_{u \in V(G)}$ where 
\begin{itemize}
\item for each $u \in V(G)$, $S_u$ is an $r$-dimensional subspace of $\mathbb{C}^d$; and
\item $S_u \perp S_v$ if and only if $uv \notin E(G)$.
\end{itemize}
A faithful orthogonal representation (as defined in Section \ref{sec-intro-defs}) generates a $(d;1)$ faithful orthogonal subspace representation, and vice versa. Further, a $(d;r)$-FOSR for a graph $G$ is a $(d;r)$-OSR for its complement $\overline{G}$, but the reverse statement is not true in general.

Now that we have defined an $r$-fold analogue of a faithful orthogonal representation, it is natural to consider a corresponding version of $\mr^+(G)$. The \emph{$r$-fold minimum positive semidefinite rank} of $G$ is
\[ \mrr{r}^+(G) = \min \{ d : G \text{ has a $(d;r)$ faithful orthogonal subspace representation} \}. \]
In particular, we have $\mrr{1}^+(G) = \mr^+(G)$, using definition \eqref{def-mrplus} of $\mr^+$; we caution the reader that this coincides with the definitions of faithful orthogonal representation and minimum positive semidefinite rank in the literature (e.g. \cite{BHH11, HLA2ch46}) if and only if $G$ has no isolated vertices.

We note that $\mrr{r}^+(G)$ is subadditive. The proof is analogous to the proof of Lemma \ref{xirsubadditive} and is omitted, as are the proofs for other results in this section that parallel those for the non-faithful case (i.e., the $\xi$-family of parameters).

\begin{Lemma} \label{mrrsubadditive}
$\mrr{r}^+$ is a subadditive function of $r$, i.e., for every graph $G$ and all $r, s \in \mathbb{N}$, 
\[ \mrr{r+s}^+(G) \leq \mrr{r}^+(G) + \mrr{s}^+(G). \]
\end{Lemma}

As in the non-faithful case, an immediate corollary relates $\mrr{r}^+$ to $\mr^+$.

\begin{Corollary} \label{mrrleqrtimesmr}
For every graph $G$ and all $r \in \mathbb{N}$,
\[ \frac{\mrr{r}^+(G)}{r} \leq \mr^+(G). \]
\end{Corollary}

For any graph $G$, we define the \emph{fractional minimum positive semidefinite rank} of $G$ as
\[ \mr^+_f(G) = \inf_r \left\{ \frac{\mrr{r}^+(G)}{r} \right\}. \]
Notice that if $G$ has a $(d;r)$ faithful orthogonal subspace representation, then $\mrr{r}^+(G) \leq d$, so $\mr_f^+(G) \leq \frac{d}{r}$.

\smallskip
We can upper bound fractional minimum positive semidefinite rank by the non-fractional version by using Corollary \ref{mrrleqrtimesmr}. Again, recall that this coincides with the literature if and only if the graph $G$ has no isolated vertices.

\begin{Corollary} \label{mrfleqmr} For every graph $G$,
\[ \mr_f^+(G) \leq \mr^+(G). \]
\end{Corollary}

Since $\mrr{r}^+(G)$ is subadditive, we have the following corollary, which follows from Lemma \ref{fgtA41} (\cite{fgt}, Lemma A.4.1).

\begin{Corollary} \label{mrf-equals-lim} For every graph $G$,
\[ \mr_f^+(G) = \lim_{r \to \infty} \frac{\mrr{r}^+(G)}{r} \, , \]
and this limit exists.
\end{Corollary}

We conclude this section with an example that gives further insight into these new parameters.

\begin{Example} \label{ex-mrr-of-P4}
{\rm
Let $r \in \mathbb{N}$ and consider the graph $G = P_4$ with $V(P_4) = \{1,2,3,4\}$ and $E(P_4) = \{12, 23, 34\}$. With $\bv{e}_i$ as the $i^{th}$ standard basis vector in $\mathbb{C}^{2r+1}$, we can verify that the following is a valid $(2r+1;r)$-FOSR for $P_4$: $S_1 = \range([\bv{e}_1, \, \bv{e}_2, \, \ldots, \, \bv{e}_{r}])$, $S_2 = \range([ \bv{e}_2, \, \bv{e}_3, \, \ldots, \, \bv{e}_{r+1}])$, $S_3 = \range([\bv{e}_{r+1}, \, \bv{e}_{r+2}, \, \ldots, \, \bv{e}_{2r}])$, $S_4 = \range([\bv{e}_{r+2}, \, \bv{e}_{r+3}, \, \ldots, \, \bv{e}_{2r+1}])$. Therefore, $\mrr{r}^+(P_4) \leq 2r+1$. Suppose that $\{Q_u\}_{u \in V(P_4)}$ is a $(2r; r)$-FOSR for $P_4$; we show that such a representation cannot exist. Since $13, 14 \notin E(P_4)$, $Q_1 \perp Q_3$ and $Q_1 \perp Q_4$. The underlying space is $\mathbb{C}^{2r}$ and each subspace $Q_i$ is $r$-dimensional, so we must therefore have $Q_3 = Q_4 = Q_1^\perp$. Now, $23 \in E(P_4)$, so $Q_2 \not\perp Q_3$, but $24 \notin E(P_4)$, so it also follows that $Q_2 \perp Q_4$. Since $Q_3 = Q_4$, this is a contradiction; thus there is no $(2r;r)$-FOSR for $P_4$, and so $\mrr{r}^+(P_4) = 2r+1$. Using the limit characterization of $\mr_f^+$, it follows that $\mr_f^+(P_4) = \lim_{r \to \infty} \frac{2r+1}{r} = 2$.
}
\end{Example}

This example demonstrates that the infimum in the definition of the fractional minimum positive semidefinite rank cannot be replaced with a minimum, even when $\mr_f^+$ is a rational number. Additionally, since $\mr^+(P_4) = 3$, the graph $G = P_4$ satisfies $\mr_f^+(G) < \mr^+(G)$.

\subsection{Faithful $\mathbit{d/r}$-projective representations}

Let $G$ be a graph and $d, r \in \mathbb{N}$ with $r \leq d$. A \emph{faithful $d/r$-projective representation}, or faithful $d/r$-representation for short, is an assignment of matrices $\{P_u\}_{u \in V(G)}$ to the vertices of $G$ such that
\begin{itemize}
\item for each $u \in V(G)$, $P_u \in \mathbb{C}^{d \times d}$, $\rank P_u  = r$, $P_u^* = P_u$, and $P_u^2 = P_u$; and
\item $P_u P_v = 0$ if and only if $uv \notin E(G)$.
\end{itemize}
A faithful $d/r$-representation for $G$ is a $d/r$-representation for $\overline{G}$, but the reverse is not necessarily true.

It is convenient to note that a $(d;r)$ faithful orthogonal subspace representation for $G$ is equivalent to a faithful $d/r$-representation. The proof is analogous to that of Proposition \ref{drOSRiffdr}; as before, we will omit such parallel proofs.

\begin{Proposition} \label{drFOSRifffaithfuldr}
A graph $G$ has a $(d;r)$ faithful orthogonal subspace representation if and only if $G$ has a faithful $d/r$-representation.
\end{Proposition}

An immediate corollary gives an alternate definition for $\mr_f^+(G)$.
\begin{Corollary} For every graph $G$,
\[ \mr^+_f(G) = \inf_{d,r} \left\{ \frac{d}{r} : G \text{ has a faithful $d/r$-representation} \right\}. \]
\end{Corollary}

\begin{Corollary} \label{xifleqmrf}
For any graph $G$ with complement $\overline{G}$,
\[ \xi_f(\overline{G}) \leq \mr_f^+(G) \leq \mr^+(G). \]
\end{Corollary}
\begin{proof}
This follows from the fact that any faithful $d/r$-representation for $G$ is also a $d/r$-\break representation for $\overline{G}$, as well as from Corollary \ref{mrfleqmr}. \qquad
\end{proof}

\subsection{Relation to positive semidefinite matrices}

In this section, we connect $(d;r)$ faithful orthogonal subspace representations to positive semidefinite matrices, thus generalizing the known results for the $r = 1$ case (when the graph in question has no isolated vertices) and connecting $\mrr{r}^+(G)$ to the rank of a positive semidefinite matrix.

We begin with some definitions. Let $G$ be a graph on $n$ vertices and suppose that $V(G) = \brackets{1}{n}$. For some $r \in \mathbb{N}$, let $A \in \mathbb{C}^{nr \times nr}$ be partitioned into an $n \times n$ block matrix $[A_{ij}]$, where $A_{ij}$ is the $r \times r$ submatrix in (block) row $i$ and (block) column $j$ of $A$. We say that the matrix $A$ \emph{$r$-fits} $G$ if $A_{ii} = I_r$ for each $i \in V(G)$ and $A_{ij} = 0$ if and only if $ij \notin E(G)$, and define the set 
\[ \Hr{r}^+(G) = \left\{ A \in \mathbb{C}^{nr \times nr} : A \mgeq 0 \text{ and $A$ $r$-fits $G$} \right\}. \]

\begin{Example} {\rm
We provide a simple example for the $r = 2$ case. Let $G = P_3$, the path on 3 vertices, with $V(G) = \{1,2,3\}$ and $E(G) = \{ 12, 23 \}$. Choosing $X = [ \bv{e}_1 \, \bv{e}_2 \, | \, \bv{e}_1 \, \bv{e}_4 \, | \, \bv{e}_3 \, \bv{e}_4]$, where $\bv{e}_j$ is the $j^{th}$ standard basis vector in $\mathbb{C}^{4}$, we can verify that

\[ A = X^* X = \left[ \begin{array}{cc|cc|cc} 
1 & 0 & 1 & 0 & 0 & 0 \\ 
0 & 1 & 0 & 0 & 0 & 0 \\ \hline
1 & 0 & 1 & 0 & 0 & 0 \\
0 & 0 & 0 & 1 & 0 & 1 \\ \hline
0 & 0 & 0 & 0 & 1 & 0 \\
0 & 0 & 0 & 1 & 0 & 1
\end{array} \right] \in \Hr{2}^+(P_3). \]
}
\end{Example}

This constructive example gives an intuitive feel for one direction of the proof of the main result of this section.

\begin{Theorem} For every graph $G$ on $n$ vertices and any $r \in \mathbb{N}$,
\[ \mrr{r}^+(G) = \min \left\{ \rank A  : A \in \Hr{r}^+(G) \right\}. \]
\end{Theorem}
\begin{proof}
Let $d = \mrr{r}^+(G)$ and let $\ell = \min \left\{ \rank A  : A \in \Hr{r}^+(G) \right\}$.

First, assume that $\{S_i \}$ is a $(d;r)$ faithful orthogonal subspace representation for $G$ and for each $i \in V(G)$ let $X_i \in \mathbb{C}^{d \times r}$ be a basis matrix for $S_i$. Define $X = [ X_1 \, | \, X_2 \, |  \cdots | \, X_n ] \in \mathbb{C}^{d \times nr}$ and let $B = X^* X \in \mathbb{C}^{nr \times nr}$. We see immediately that $B \mgeq 0$ and $\rank B = \rank X  \leq d$. Partitioning $B$ into an $n \times n$ block matrix with blocks $[B_{ij}]$ of size $r \times r$, we have $B_{ij} = X_i^* X_j$. Since $S_i \perp S_j$ if and only if $X_i^* X_j = 0$, we have $B_{ij} = 0$ if and only if $S_i \perp S_j$, which occurs if and only if $ij \notin E(G)$. Additionally, since $X_i$ has orthonormal columns, we have $B_{ii} = I_r$ for each $i$. Therefore, $B \in \Hr{r}^+(G)$, so $\min \left\{ \rank A  : A \in \Hr{r}^+(G) \right\} \leq \rank B \leq d = \mrr{r}^+(G)$.

For the reverse inequality, suppose that $B \in \Hr{r}^+(G)$ and $\rank B  = \ell$. Then there exists a matrix $X \in \mathbb{C}^{\ell \times nr}$ such that $B = X^* X$.  Partition $B$ into $r \times r$ blocks $[B_{ij}]$ and partition $X$ into $\ell \times r$ blocks as $X = [ X_1 \, | \, X_2 \, | \cdots | \, X_n ]$. For each vertex $i \in V(G)$, let $S_i = \range(X_i) \subseteq \mathbb{C}^\ell$. Since $X_i^* X_i = I_r$, we have $\rank X_i  = r$, so each $S_i$ is an $r$-dimensional subspace of $\mathbb{C}^\ell$. Additionally, $X_i^*X_j = B_{ij} = 0$ if and only if $ij \notin E(G)$, so $S_i \perp S_j$ if and only if $ij \notin E(G)$. Therefore, $\{S_i\}$ is an $(\ell;r)$ faithful orthogonal subspace representation for $G$, so $\mrr{r}^+(G) \leq \ell = \min \left\{ \rank A  : A \in \Hr{r}^+(G) \right\}$ and thus equality holds. \qquad
\end{proof}

This matrix-based representation is a powerful theoretical tool that allows us to simplify the proofs of some properties of $r$-fold minimum positive semidefinite rank, as well as to more clearly draw parallels to the existing and well-established $r=1$ case (although again, the connection to the literature requires that the graph in question has no isolated vertices).

The condition that $A_{ii} = I_r$ if $A$ $r$-fits a graph $G$ is a strong one, so we conclude this section with a weaker condition that will be used to further simplify proofs without sacrificing utility. We say that $A$ \emph{weakly $r$-fits} $G$ if $A_{ii}$ is a diagonal matrix with strictly positive diagonal entries for each $i \in V(G)$ and $A_{ij} = 0$ if and only if $ij \notin E(G)$. Clearly, any matrix that $r$-fits $G$ also weakly $r$-fits $G$.

\begin{Remark} {\rm
Suppose that $A$ weakly $r$-fits a graph $G$ and let $D = D_1 \oplus \cdots \oplus D_n$, where each $D_i$ is the inverse of the positive square root of $A_{ii}$, i.e., $D_i = A_{ii}^{-\frac{1}{2}}$. Then the matrix $B = DAD$ $r$-fits $G$, since $D$ is a diagonal matrix with strictly positive diagonal entries, so multiplication by $D$ does not change the zero pattern of $A$. Further, $\rank B = \rank A$, since $D$ has full rank.
}
\end{Remark}

This remark yields an immediate corollary to the previous theorem.

\begin{Corollary} For every graph $G$ on $n$ vertices and any $r \in \mathbb{N}$,
\[ \mrr{r}^+(G) = \min \left\{ \rank A : A \in \mathbb{C}^{nr \times nr}, \, A \mgeq 0 \text{ and $A$ weakly $r$-fits $G$} \right\}. \]
\end{Corollary} 

\subsection{Properties of mr$\mathbit{_{[r]}^+(G)}$ and mr$\mathbit{_f^+(G)}$}

In this section, we prove numerous results regarding properties of $r$-fold and fractional minimum positive semidefinite rank, many of which extend known properties of $\mr^+$ to the new parameters.

\begin{Observation} \label{mrr-geq-alpha}
For every graph $G$ and all $r \in \mathbb{N}$, $\mrr{r}^+(G) \geq r \cdot \alpha(G)$.
\end{Observation}

\begin{Proposition} \label{mrr-subgraph}
Let $r \in \mathbb{N}$ and let $H$ be an induced subgraph of $G$. Then $\mrr{r}^+(H) \leq \mrr{r}^+(G)$.
\end{Proposition}
\begin{proof}
For any $u,v \in V(H)$, $uv \in E(H)$ if and only if $uv \in E(G)$, since $H$ is induced. Therefore any $(d;r)$ faithful orthogonal subspace representation for $G$ provides a $(d;r)$ faithful orthogonal subspace representation for $H$, and the result follows immediately. \qquad
\end{proof}

\begin{Proposition} \label{mrr-dcup}
If $G = \bigdcup_{i=1}^t G_i$ for some graphs $\{G_i\}_{i=1}^t$, then $\mrr{r}^+(G) = \sum_{i=1}^t \mrr{r}^+(G_i)$ for each $r \in \mathbb{N}$.
\end{Proposition}

\begin{proof}
Suppose that $V(G) = \brackets{1}{n}$ and that $|V(G_i)| = n_i$ for $i = 1, 2, \ldots, t$. Further assume that $V(G_i) = \brackets{1 + \sum_{j=1}^{i-1} n_j}{\sum_{j=1}^i n_j}$, so that if $A \in \Hr{r}^+(G)$, then $A = A_1 \oplus A_2 \oplus \cdots \oplus A_t$, where $A_i \in \Hr{r}^+(G_i)$ for each $i$. Note that $\rank A  = \sum_{i=1}^t \rank A_i$. We therefore have
\begin{align*}
\mrr{r}^+(G)
	&=	\min \left\{ \rank A  : A \in \Hr{r}^+(G) \right\} \\
	&=	\min \left\{ \sum_{i=1}^t \rank A_i  : A_i \in \Hr{r}^+(G_i) \text{ for each $i$} \right\} \\
	&=	\sum_{i=1}^t \min \left\{ \rank A_i  : A_i \in \Hr{r}^+(G_i) \right\} \\
	&=	\sum_{i=1}^t \mrr{r}^+(G_i).  \qedhere
\end{align*}
\end{proof}

\begin{Theorem}
If $G = \bigcup_{i=1}^t G_i$ for some graphs $\{G_i\}_{i=1}^t$, then $\mrr{r}^+(G) \leq \sum_{i=1}^t \mrr{r}^+(G_i)$ for each $r \in \mathbb{N}$.
\end{Theorem}

\begin{proof}
We prove the result for the case where $t = 2$ and note that recursive application of this case will prove the more general one.

Let $V(G) = \brackets{1}{n}$ where $n > 0$ and assume that $V(G_1) \setminus V(G_2) = \brackets{1}{n_1}$, $V(G_1) \cap V(G_2) = \brackets{n_1+1}{n_1+c}$, and $V(G_2) \setminus V(G_1) = \brackets{n_1+c + 1}{n_1+c+n_2}$, where $n_1, n_2, c \geq 0$ (it is not assumed that each of these is strictly nonzero). Note that $n = n_1 + c + n_2$, and this ordering asserts that the first $n_1$ vertices (enumerating in the natural order) lie exclusively in $G_1$, the next $c$ are common to both graphs, and the last $n_2$ lie exclusively in $G_2$.

For $i = 1,2$, let $\mrr{r}^+(G_i) = d_i$ and let $A_i \in \Hr{r}^+(G_i)$ be chosen so that $\rank A_i  = d_i$. Notice that $A_1 \in \mathbb{C}^{(n_1+c)r \times (n_1+c)r}$ has its rows and columns indexed by $V(G_1) = \brackets{1}{n_1+c}$ and $A_2 \in \mathbb{C}^{(n_2+c)r \times (n_2+c)r}$ has its rows and columns indexed by $V(G_2) = \brackets{n_1+1}{n}$.

Let
\[ \hat A_1 = \left[ \begin{array}{cc} A_1 & 0 \\ 0 & 0 \end{array} \right] \in \mathbb{C}^{nr \times nr}, \qquad \hat A_2 = \left[ \begin{array}{cc} 0 & 0 \\ 0 & A_2 \end{array} \right] \in \mathbb{C}^{nr \times nr} \]
and define $A = \hat A_1 + \beta \hat A_2 \in \mathbb{C}^{nr \times nr}$, where $\beta > 0$ is chosen so that if $A$, $\hat A_1$, and $\hat A_2$ are partitioned into $n \times n$ block matrices with block size $r \times r$, then $A_{ij} = 0$ if and only if $(\hat A_1)_{ij} = 0$ and $(\hat A_2)_{ij} = 0$ (i.e., no cancellation of an entire block occurs).

Since $A$ is a positive linear combination of positive semidefinite matrices, $A \mgeq 0$, and by our choice of $\beta$ we see that $A$ weakly $r$-fits $G$. Therefore,
\[ \mrr{r}^+(G) \leq \rank A  \leq \rank \hat A_1 + \rank \hat A_2 = d_1 + d_2 = \mrr{r}^+(G_1) + \mrr{r}^+(G_2). \qedhere \]
\end{proof}

All of the results we have proven for $r$-fold minimum positive semidefinite rank can be extended to results for fractional minimum positive semidefinite rank. The proof is analogous to that of Theorem \ref{xifproperties} and is omitted.

\begin{Theorem} For every graph $G$:
\begin{enumerate}[i)]
\item $\mr_f^+(G) \geq \alpha(G)$.
\item If $H$ is an induced subgraph of $G$, then $\mr_f^+(H) \leq \mr_f^+(G)$.
\item If $G = \bigdcup_{i=1}^t G_i$ for some graphs $\{ G_i \}_{i=1}^t$, then $\mr_f^+(G) = \sum_{i=1}^t \mr_f^+(G_i)$.
\item If $G = \bigcup_{i=1}^t G_i$ for some graphs $\{ G_i \}_{i=1}^t$, then $\mr_f^+(G) \leq \sum_{i=1}^t \mr_f^+(G_i)$.
\end{enumerate}
\end{Theorem}

\smallskip
Let $G$ be a connected graph of order at least two. A standard technique for computing the minimum positive semidefinite rank of $G$ is cut-vertex reduction \cite{BHH11, HLA2ch46, Hol09}: Suppose that $v \in V(G)$ is a cut-vertex and $(G - v)$ has connected components $\{H_i\}_{i=1}^t$. For each $i$, let $G_i$ be the subgraph of $G$ induced by the union of the vertices of $H_i$ with $v$, that is, $G_i = G[V(H_i) \cup \{v\}]$. Then $\mr^+(G) = \sum_{i=1}^t \mr^+(G_i)$. Unfortunately, this technique does not carry over to the $r$-fold case when $r > 1$, as the following example shows.

\begin{Example} {\rm
Consider the graph $G = P_4$, the path on 4 vertices, with $V(G) = \{x,y,v,z\}$ in path order; recall from Example \ref{ex-mrr-of-P4} that $\mrr{r}^+(G) = 2r+1$ for any $r \in \mathbb{N}$. Taking $v$ as a cut-vertex, we have $G_1 = P_3$ with $V(G_1) = \{x,y,v\}$ and $G_2 = P_2$ with $V(G_2) = \{v, z\}$. Fix $r > 1$. Since $\alpha(G_1) = 2$, any valid $(d;r)$-FOSR for $G_1$ must have $d \geq 2r$. Further, it is easy to see that $\mr^+(G_1) = 2$, so $4 \leq \mrr{r}^+(G_1) \leq 2 \cdot \mr^+(G_1) = 2r$. Hence equality holds and $\mrr{r}^+(G_1) = 2r$. Next, since $\mr^+(G_2) = 1$ and $d \geq r$ for any valid $(d;r)$-FOSR, we have $r \leq \mrr{r}^+(G_2) \leq r \cdot \mr^+(G_2) = r$, so $\mrr{r}^+(G_2) = r$. Hence if $r > 1$, then $\mrr{r}^+(G) = 2r+1 < 2r + r = \mrr{r}^+(G_1) + \mrr{r}^+(G_2)$, so cut-vertex reduction does not apply.
}
\end{Example}

\subsection{Fractional minimum positive semidefinite rank and projective rank} \label{subsec-equality}

Recall that any $(d;r)$-FOSR for $G$ is a $(d;r)$-OSR for $\overline{G}$, but the reverse statement does not apply in general. It thus follows that $\xir{r}(\overline{G}) \leq \mrr{r}^+(G)$ for any graph $G$ and $r \in \mathbb{N}$, and the next example demonstrates that this inequality can be strict.

\begin{Example} { \rm
Consider the graph $G = P_4$ with $V(P_4) = \{1,2,3,4\}$ and $E(P_4) = \{12, 23, 34 \}$ and fix $r \in \mathbb{N}$. Since $\omega(P_4) = 2$, we have $\xir{r}(P_4) \geq 2r$. With $\bv{e}_i$ as the $i^{th}$ standard basis vector for $\mathbb{C}^{2r}$, it is easy to verify that the following is a $(2r;r)$-OSR for $P_4$: $S_1 = S_3 = \range([\bv{e}_1, \, \bv{e}_2, \, \ldots, \, \bv{e}_{r}])$, $S_2 = S_4 = \range([ \bv{e}_{r+1}, \, \bv{e}_{r+2}, \, \ldots, \, \bv{e}_{2r}])$. Therefore, $\xir{r}(P_4) = 2r$. Since $\overline{P_4} = P_4$ and $\mrr{r}^+(P_4) = 2r+1$ (Example \ref{ex-mrr-of-P4}), we have $2r = \xir{r}(\overline{P_4}) < \mrr{r}^+(P_4) = 2r+1$.
}
\end{Example}

Recall from Corollary \ref{xifleqmrf} that $\xi_f(\overline{G}) \leq \mr_f^+(G)$ for any graph $G$. While strict inequality may hold in the $r$-fold case for an arbitrary graph $G$, we now demonstrate that equality always holds in the ``fractional case" for any graph $G$. For this result, we require the following lemma.

\begin{Lemma} \label{repswithineps}
Let $G$ be a graph with complement $\overline{G}$. Let $\{P_u\}_{u \in V(G)}$ be a $d/r$-representation for $\overline{G}$ and let $\{R_u\}_{u \in V(G)}$ be a faithful $b/1$-representation for $G$. Then for any $k \in \mathbb{N}$, $G$ has a faithful $(kd+b)/(kr+1)$-representation $\{Q_u\}_{u \in V(G)}$. Further, given any $\eps > 0$, $k$ can be chosen such that $\left| \frac{d}{r} - \frac{kd+b}{kr+1} \right| < \eps$, i.e., the value of the faithful representation $\{Q_u\}$ for $G$ is within $\eps$ of the value of the (non-faithful) representation $\{P_u\}$ for $\overline{G}$.
\end{Lemma}
\begin{proof}
Since $\{P_u\}$ is a $d/r$-representation for $\overline{G}$, we have $P_u \in \mathbb{C}^{d \times d}$ with $\rank P_u = r$ for each $u \in V(\overline{G}) = V(G)$, and $P_u P_v = 0$ if $uv \in E(\overline{G})$, so $P_u P_v = 0$ if $uv \notin E(G)$.

Let $\eps > 0$ be arbitrary and choose $k > \left( \frac{|d-rb|}{r^2 \eps} - \frac{1}{r} \right)$.

\smallskip
For each vertex $u \in V(G)$, let $Q_u \in \mathbb{C}^{(kd+b) \times (kd+b)}$ be the block diagonal matrix constructed from $k$ copies of $P_u$ and one copy of $R_u$, i.e.,
\[ Q_u = \left(\bigoplus_{i=1}^k P_u \right) \oplus R_u. \]
We see immediately that $\rank Q_u = kr+1$, and since $P_u$ and $R_u$ are projectors, so is $Q_u$. Since $P_u P_v = 0$ if $uv \notin E(G)$ and $R_u R_v = 0$ if \emph{and only if} $uv \notin E(G)$, we conclude that $Q_u Q_v = 0$ if and only if $uv \notin E(G)$. Therefore, $\{Q_u\}_{u \in V(G)}$ is a faithful $(kd+b)/(kr+1)$-representation for $G$, which verifies the first claim.

By choice of $k$, we have $kr+1 > \frac{|d-rb|}{r \eps}$. Consider
\begin{align*}
\left| \frac{d}{r} - \frac{kd+b}{kr+1} \right|
	&=	\left| \frac{d(kr+1) - r(kd+b)}{r(kr+1)} \right| \\
	&=	\frac{|d-rb|}{r} \cdot \frac{1}{kr+1} \\
	&<	\frac{|d-rb|}{r} \cdot \frac{r\eps}{|d-rb|} \\
	&=	\eps,
\end{align*}
which verifies the second claim. \qquad
\end{proof}

It was previously noted that any faithful $d/r$-representation for $G$ is also $d/r$-representation for $\overline{G}$. Lemma \ref{repswithineps} is a partial converse in the sense that, given any $d/r$-representation for $\overline{G}$, we can construct a faithful $d_1/r_1$-representation for $G$ such that the two representations have essentially the same value. This yields the next result.

\begin{Theorem}
For every graph $G$ with complement $\overline{G}$,
\[ \xi_f(\overline{G}) = \mr_f^+(G). \]
\end{Theorem}
\begin{proof}
Let
\[ R = \left\{ \frac{d}{r} : \overline{G} \text{ has a $d/r$-representation} \right\}, \]
\[ F = \left\{ \frac{d}{r} : G \text{ has a faithful $d/r$-representation} \right\}. \]
For any $\frac{d}{r} \in R$ and $\eps > 0$, Lemma \ref{repswithineps} asserts that there exists some $\frac{d_1}{r_1} \in F$ such that $\left| \frac{d}{r} - \frac{d_1}{r_1} \right| < \eps$. It follows that $\inf R = \inf F$, i.e., $\xi_f(\overline{G}) = \mr_f^+(G)$. \qquad
\end{proof}


\section*{Acknowledgements}

Some of this work was done while Leslie Hogben was a general member of the Institute for Mathematics and its Applications (IMA) and during a week-long visit of Kevin Palmowski to IMA; they thank IMA both for financial support (from NSF funds) and for providing a wonderful collaborative research environment.

David E. Roberson is supported in part by the Singapore National Research Foundation under NRF RF Award No. NRF-NRFF2013-13.

Simone Severini is supported by the Royal Society and EPSRC.


\end{document}